\documentclass[reqno,12pt]{amsart}
\usepackage[english]{babel}
\usepackage{amsmath,amsfonts,amsthm,amssymb,amsbsy,graphicx,amscd,hyperref,enumerate,enumitem}
\usepackage{mathtools,url}
\usepackage{tikz-cd}
\usetikzlibrary{arrows}
\usepackage[active]{srcltx}

\font\script=rsfs10 at 12pt
\def\H{{\mbox{\script H}\,\,}}
\def\N{\mathbb{N}}
\def\O{\Omega}
\def\R{\mathbb{R}}
\def\S{\mathbb{S}}

\def\F{\mathcal{F}}

\def\M{\mathcal{M}}

\newcommand{\be}{\begin{equation}}
\newcommand{\ee}{\end{equation}}
\newcommand{\bib}[4]{\bibitem{#1}{\sc#2: }{\it#3. }{#4.}}
\newcommand{\cp}{\mathop{\rm cap}\nolimits}
\def\comp{\subset\subset}
\def\bal{\begin{aligned}}
\def\eal{\end{aligned}}

\numberwithin{equation}{section}
\theoremstyle{plain}
\newtheorem{theo}{Theorem}[section]
\newtheorem{lemm}[theo]{Lemma}

\newtheorem{defi}[theo]{Definition}

\theoremstyle{remark}
\newtheorem{rema}[theo]{Remark}

\newcommand{\s}{\sigma}

\title[Continuous Steiner symmetrization and Blaschke-Santal\'o diagrams]{An application of the continuous Steiner symmetrization to Blaschke-Santal\'o diagrams}

\author[G. Buttazzo]{Giuseppe Buttazzo}
\author[A. Pratelli]{Aldo Pratelli}

\date{}

\begin{document}

\maketitle
{\it Dedicated to Enrique Zuazua for his 60th birthday}

\begin{abstract}
In this paper we consider the so-called procedure of {\it Continuous Steiner Symmetrization}, introduced by Brock in~\cite{bro95,bro00}. It transforms every domain $\Omega\comp\R^d$ into the ball keeping the volume fixed and letting the first eigenvalue and the torsion respectively decrease and increase. While this does not provide, in general, a $\gamma$-continuous map $t\mapsto\O_t$, it can be slightly modified so to obtain the $\gamma$-continuity for a $\gamma$-dense class of domains $\O$, namely, the class of polyedral sets in $\R^d$. This allows to obtain a sharp characterization of the Blaschke-Santal\'o diagram of torsion and eigenvalue.
\end{abstract}



\section{Introduction}

The question of making a given domain $\O\subset\R^d$ more and more round, keeping constant its measure, up to reach a ball, was first considered by Steiner, who proposed to use successive symmetrizations through different hyperplanes. More precisely, given a domain $\Omega\subset\R^d$ and a direction $\nu\in \S^{d-1}$, the \emph{Steiner symmetrization of $\Omega$ with respect to $\nu$} is defined as 
\[
\Omega^*_\nu = \bigg\{ x \in \R^d:\, | x \cdot \nu | \leq \frac{\varphi\big(\pi(x)\big)}2\bigg\}\,,
\]
where $\pi(x)=x - \nu (x\cdot \nu)$ is the projection of any point $x\in \R^d$ onto the hyperplane orthogonal to $\nu$ and where, for each $y$ in this hyperplane,
\[
\varphi(y) = \H^1\big(\Omega\cap \pi^{-1}(y)\big)
\]
is the length of the $y$-section of $\Omega$. The set $\Omega^*_\nu$ has the same volume of $\Omega$ and is a bit ``nicer'', in particular it is symmetric through the hyperplane orthogonal to $\nu$. It is not difficult to guess that, repeating this symmetrization through a sequence of hyperplanes with properly chosen directions, one obtains a sequence $\Omega_n$ of sets, all with the same measure, which $\gamma$-converge as $n\to\infty$ to a ball. The interest in this symmetrization procedure consists in the fact that along the sequence $\O_n$ several quantities improve, and become asymptotically optimal as $n\to\infty$. In particular we are interested in the following quantities.

$\bullet\ $ The {\it first eigenvalue} $\lambda(\O)$ of the Laplace operator $-\Delta$ with Dirichlet conditions on $\partial\O$, defined as the smallest number $\lambda$ providing a nonzero solution to the PDE
\begin{align*}
-\Delta u=\lambda u\text{ in }\O\,, &&  u\in H^1_0(\O)\,,
\end{align*}
or equivalently through the minimization of the Rayleigh quotient
\[
\lambda(\O)=\min\left\{\Big[\int_\O|\nabla u|^2\,dx\Big]\Big[\int_\O|u|^2\,dx\Big]^{-1}\ :\ u\in H^1_0(\O)\setminus\{0\}\right\}\,.
\]
An important bound for $\lambda(\O)$ is the {\it Faber-Krahn inequality},
\be\label{fbineq}
|\O|^{2/d}\lambda(\O)\ge|B|^{2/d}\lambda(B)
\ee
where $B$ is any ball in $\R^d$.\\
$\bullet\ $ The {\it torsional rigidity} $T(\O)$, defined as $\int_\O u_\O\,dx$, where $u_\O$ is the unique solution of the PDE
\begin{align*}
-\Delta u=1\text{ in }\O\, && u\in H^1_0(\O)\,,
\end{align*}
or equivalently through the maximization problem
\[
T(\O)=\max\left\{\Big[\int_\O u\,dx\Big]^2\Big[\int_\O|\nabla u|^2\,dx\Big]^{-1}\ :\ u\in H^1_0(\O)\setminus\{0\}\right\}\,,
\]
where the maximum is reached by $u_\O$ itself. Also for $T(\O)$ an important inequality is true, that is, the {\it Saint-Venant inequality}
\be\label{svineq}
|\O|^{-(d+2)/d}T(\O)\le|B|^{-(d+2)/d}T(B)
\ee
where $B$ is any ball in $\R^d$. 

The inequalities~(\ref{fbineq}) and~(\ref{svineq}) ensure that balls minimize the first eigenvalue, and maximize the torsional rigidity, among sets of given volume. It is easy to verify that the quantities above fulfill the following scaling properties:
\begin{align*}
\lambda(s\O)=s^{-2}\lambda(\O)\,, && T(\s\O)=s^{d+2}T(\O)\,.
\end{align*}

It is well-known (see for instance~\cite{ALT91}) that the Steiner symmetrization decreases the first eigenvalue and increases the torsional rigidity, that is, for every set $\Omega\subset \R^d$ and direction $\nu\in\S^{d-1}$ one has
\begin{align*}
\lambda(\Omega^*_\nu)\leq \lambda(\Omega)\,, &&
T(\Omega^*_\nu)\geq T(\Omega)\,,
\end{align*}
so that for the sequence $\Omega_n$ defined above one has that $\lambda(\O_n)$ (resp. $T(\O_n)$) decreases (resp., increases) with respect to $n$, and converges to $\lambda(B)$ (resp., $T(B)$), being $B$ any ball with $|B|=|\O|$.

A natural question is whether the discrete approximation can be replaced by a continuous one. More precisely, one would like to have a family $\O_t$, with $t\in[0,1]$, such that $\O_0=\O$, $\O_1=B$ and such that $t\mapsto \lambda(\O_t)$ and $t\mapsto T(\O_t)$ are respectively continuously decreasing and continuously increasing. In addition, the family of sets should be continuous with respect to the $\gamma$-convergence, which is the natural convergence for variational problems, and that we briefly recall in Section~\ref{ssgc}. As described above, successive Steiner symmetrizations allow to pass from a generic set to the ball, hence it is enough to construct a continuous approximation which transforms a set $\Omega$ into its Steiner symmetrization $\Omega^*_\nu$.\par

An explicit construction of a family $\O_t$ transforming the set $\Omega$ into its Steiner symmetrization $\Omega^*_\nu$, called \emph{continuous Steiner symmetrization}, was proposed by Brock in~\cite{bro95}, see also~\cite{bro00}. Previously, other constructions had been proposed, see for instance~\cite{BLL74,K85}. With the Brock construction, that we will briefly describe in Section~\ref{ssbrock}, the quantities $\lambda(\Omega_t)$ and $T(\Omega_t)$ are respectively decreasing and increasing, but they are not continuous, in particular they are both continuous from the left, and respectively upper and lower semicontinuous from the right (see for instance~\cite{BBF99}). The full $\gamma$-continuity of the Brock construction, which implies also the continuity of first eigenvalue and torsional rigidity, only holds on restricted classes of domains, as for instance the class of {\it convex} domains.

On the other hand, a $\gamma$-continuous symmetrization $(\O_t)$ which makes $\lambda(\O_t)$ and $T(\O_t)$ continuously decreasing and increasing would be very useful in several situations. In this paper we show that a simple modification of the Brock construction is enough to define such a symmetrization for the class of polyhedral domains, which are known to be $\gamma$-dense among all domains. Despite the fact that this is a very specific class, the result is enough to prove that the Blaschke-Santal\'o diagram corresponding to the pair $\big(\lambda(\O),T(\O)\big)$ is between two graphs. Several other estimates for various kinds of quantities depending on a domain $\O$ are available in the recent literature; we refer the interested reader to \cite{BB20, BBV15, BBP21, flxx, luzu20} and to references therein.

Let us be more precise. Calling $B$ any ball in $\R^d$, for every domain $\Omega\subset\R^d$ we define the quantities
\begin{align*}
x_\O=\frac{|B|^{2/d}\lambda(B)}{|\O|^{2/d}\lambda(\O)}\,, && y_\O=\frac{|B|^{(d+2)/d}T(\O)}{|\O|^{(d+2)/d}T(B)}\,,
\end{align*}
which are respectively the reciprocal of the first eigenvalue $\lambda(\O)$ and the torsional rigidity $T(\O)$, suitably rescaled so to be in the interval $[0,1]$. The \emph{Blaschke-Santal\'o diagram} is the subset of $\R^2$ given by
\[
E=\left\{(x,y)\in\R^2\ :\ x=x_\O,\ y=y_\O\text{ for some domain }\O\right\}\,.
\]

Our two main results are then the following.

\begin{theo}\label{main}
For every polyhedron $\Omega\subset\R^d$ there exists a $\gamma$-continuous map $[0,1]\ni t \mapsto \Omega_t\subset\R^d$ such that every set $\Omega_t$ has the same measure, $\Omega_0=\Omega$, $\Omega_1$ is a ball, and the quantities $t\mapsto \lambda(\Omega_t)$ and $t\mapsto T(\Omega_t)$ are respectively continuously decreasing and continuously decreasing.
\end{theo}

\begin{theo}\label{bsgraph}
There exists an increasing function $h:[0,1]\to[0,1]$ such that the Blaschke-Santal\'o diagram $E$ coincides with the region of $[0,1]\times[0,1]$ between the two curves
\begin{align*}
y=x^{(d+2)/2} && \text{and} && y=h(x)\,.
\end{align*}
More precisely,
\begin{equation}\label{openincl}
\Big\{\hspace{-2pt}(x,y)\in [0,1]^2:  x^{(d+2)/2}< y < h(x) \hspace{-3pt}\Big\}\hspace{-3pt}\subseteq E \subseteq \hspace{-3pt}\Big\{\hspace{-3pt}(x,y)\in [0,1]^2: x^{(d+2)/2}\leq y \leq h(x) \hspace{-2pt}\Big\}\,.
\end{equation}
In addition, for every $x\in[0,1]$ the function $h$ satisfies 
\begin{equation}\label{estih}
x^{(d+2)/2}\Big(\big[x^{-d/2}]+\big(x^{-d/2}-\big[x^{-d/2}])^{(d+2)/d}\Big)\le h(x)\le\frac{xd(d+2)^2}{2xd+(d+2)\lambda(B)}\,,
\end{equation}
where $[\cdot]$ denotes the integer part, and $B$ is a ball of radius $1$.
\end{theo}

The approach we use to obtain Theorem~\ref{bsgraph} is rather general. Namely, we show that $E$ is ``downward and rightward convex''. More precisely, for every $(x_0,y_0)\in E$ we prove that all the points $(x,y)\in (x_0,1)\times (0,y_0)$ with $y \geq x^{(d+2)/2}$ belong to $E$. In the proof of this convexity property the $\gamma$-continuous Steiner symmetrization for polyhedra is crucial and the characterization of the structure of the set $E$ could be of great help in the analysis of several shape optimization problems. We briefly discuss the limit cases in the inclusions~(\ref{openincl}) in the final Remark~\ref{remoi}.

The plan of the paper is the following. In Section~\ref{ssgc} and in Section~\ref{ssbrock} we quickly describe the $\gamma$-convergence and the continuous Steiner symmetrization of Brock. Then, in Section~\ref{pfmain} and in Section~\ref{ssantalo} we prove respectively Theorem~\ref{main} and Theorem~\ref{bsgraph}.


\section{The $\gamma$-convergence\label{ssgc}}

In this section we recall the definition of $\gamma$-convergence, together with its main properties. For a more detailed analysis we refer to the book~\cite{bubu05}. For simplicity we always assume that all the domains we consider are contained in a fixed bounded set $D\subset\R^d$, which makes no difference for our purposes.

\begin{defi}\label{dgamma}
We say that a sequence $\{\O_n\}$ of open sets $\gamma$-converges to the open set $\O$ if for every right-hand side $f\in H^{-1}(D)$ the solutions $u_n$ of the PDEs
\begin{align*}
-\Delta u_n=f\text{ in }\O_n\,, && u_n\in H^1_0(\O_n)\,,
\end{align*}
each extended by zero on $D\setminus\O_n$, converge weakly in $H^1_0(D)$ to the solution $u$ of
\begin{align*}
-\Delta u=f\text{ in }\O\,, && u\in H^1_0(\O)\,.
\end{align*}
\end{defi}

We summarize here below the main properties of the $\gamma$-convergence. We refer to~\cite{bubu05} for all the details, properties, and proofs. 

\begin{enumerate}[label=(\arabic*)]

\item The $\gamma$-convergence can be defined in a similar way for {\it quasi-open} sets $\O\subset D$ or more generally for {\it capacitary measures} $\mu$ confined into $D$ (that is $\mu=+\infty$ outside $D$). For a capacitary measure $\mu$ the corresponding PDE is written as
\begin{align*}
-\Delta u+\mu u=f\text{ in }D\,, && u\in H^1_0(D)\cap L^2_\mu(D)\,,
\end{align*}
and has to be intended it in the weak sense, that is, $u\in H^1_0(D)\cap L^2_\mu(D)$ and
\begin{align*}
\int_D\nabla u\nabla\phi\,dx+\int_D u\phi\,d\mu=\langle f,\phi\rangle && \forall\phi\in H^1_0(D)\cap L^2_\mu(D)\,.
\end{align*}

\item The space $\M$ of capacitary measures above, endowed with the $\gamma$-convergence, is a compact space.

\item Open sets or more generally quasi-open sets belong to $\M$; for a given domain $\O$ the element of $\M$ representing it is the measure defined for all Borel sets $E\subset D$ as
\[
\infty_{\O^c}(E)=\begin{cases}
0&\text{if }\cp(E\cap\O)=0\\
+\infty&\text{otherwise.}
\end{cases}
\]

\item In Definition~\ref{dgamma} requiring the convergence of the solutions $u_n$ to $u$ for every right-hand side $f$ is equivalent to require the convergence $u_n\to u$ only for $f\equiv1$ and in the $L^2(D)$ sense. In particular, calling $u_\mu$ the solution of the PDE $-\Delta u + \mu u =1$ in $H^1_0(D)\cap L^2_\mu(D)$, the quantity
\[
d_\gamma(\mu_1,\mu_2)=\|u_{\mu_1}-u_{\mu_2}\|_{L^2(D)}
\]
is a distance on the space $\M$ of capacitary measures, which is equivalent to $\gamma$-convergence, and so $\M$ endowed with the distance $d_\gamma$ is a compact metric space.

\item Several subclasses of $\M$ are dense with respect to the $\gamma$-convergence. For instance:
\begin{enumerate}[label=(\roman*)]
\item the class of measures $a(x)\,dx$ with $a\ge0$ and smooth;
\item the class of smooth domains $\O\subset D$;
\item the class of polyedral domains $\O\subset D$.
\end{enumerate}
\item The first eigenvalue $\lambda(\O)$ (as well as all the other eigenvalues $\lambda_k(\O)$) and the torsional rigidity $T(\O)$ are continuous with respect to the $\gamma$-convergence.
\end{enumerate}


\section{The continuous Steiner symmetrization\label{ssbrock}}

In this section we describe the continuous Steiner symmetrization studied by Brock in~\cite{bro95,bro00}. As described in the introduction, this is a path of open sets $\Omega_t$ which start from a given open set $\Omega_0=\Omega$ and end with the Steiner symmetral $\Omega_\infty=\Omega^*_\nu$ of $\Omega$ with respect to a given direction $\nu\in\S^{d-1}$. In this construction the variable $t$ ranges from $0$ to $+\infty$, while in Theorem~\ref{main} we preferred to use $t\in [0,1]$, this is clearly only a matter of taste and does not make any real difference.

In order to describe this symmetrization, the important issue is to discuss the one-dimensional case. Let us start assuming that $\Omega=(a,b)$ is an open segment in $\R$. In this case, for every $t$ the set $\Omega_t$ is again a segment $(a_t,b_t)$ of length $b_t-a_t=b-a$, which moves towards right with velocity $(b_t+a_t)/2$. In other words, the position of the barycenter $m_t=(b_t+a_t)/2$ is given by $e^{-t} m_0$, and in particular $\Omega_\infty=\big(-(a+b)/2,(a+b)/2\big)$ is the Steiner symmetral of $\Omega$.\par

Let us now assume that $\Omega\subseteq\R$ is given by a finite union of open segments with disjoint closures. In this case, for small $t$ each of the segments moves according with the above rule. There is then a smallest time $t_1>0$ when two consecutive segments meet, so in particular $\Omega_{t_1}$ is given by a finite union of segments, and (at least) two of them have a common endpoint. Let us call $\Omega_{t_1}^+ ={\rm Int}\big(\overline{\Omega_{t_1}}\big)$, that is, we add to the set $\Omega_{t_1}$ the common endpoints. The set $\Omega_{t_1}^+$ is then a finite union of open segments with disjoint closures, and for $t>t_1$ with small difference $t-t_1$ we can define $\Omega_t = \big(\Omega_{t_1}^+\big)_{t-t_1}$. Again, there is a smallest time $t_2>t_1$ when two consecutive segments meet, and so on. After a finite number of junctions, the set $\Omega_t$ is then remained a single segment, and then we leave it evolve to the symmetric segment $\Omega_\infty$ as already described.\par

As shown by Brock, there is a general rule which works for all the open subsets of $\R$, and which reduces to the one depicted above in the case of finitely many segments.\par

The construction in $\R^d$ is basically one-dimensional. Calling, for every $y\in \R^d$ orthogonal to the direction $\nu$, $\Omega^y$ the $y$-section of $\Omega$, made by all points $x$ of $\Omega$ such that $y-x$ is parallel to $\nu$, one simply defines $\Omega_t$ the set such that, for every $y$, $(\Omega_t)^y=(\Omega^y)_t$. As shown in~\cite{bro95,bro00,BBF99,BH00}, the family of sets $\Omega_t$ has various properties. They are all sets with the same measure, being $\Omega_0=\Omega$ and $\Omega_\infty=\Omega^*_\nu$. In addition, the first eigenvalue $\lambda(\Omega_t)$ and the torsional rigidity $T(\Omega_t)$ are respectively decreasing and increasing with respect to $t$. More precisely, they are both continuous from the left, and they can have jumps from the right. One can say even more, that is, if $s\nearrow t$ then the sets $\Omega_s$ are $\gamma$-converging to $\Omega_t$.\par

\begin{figure}[htbp]
\begin{tikzpicture}
\fill[blue!25!white] (0,5) -- (1,5) -- (1,3) arc(180:360:.5) -- (2,5) -- (3,5) -- (3,3) arc(360:180:1.5) -- (0,5);
\draw[line width=1] (0,5) -- (1,5) -- (1,3) arc(180:360:.5) -- (2,5) -- (3,5) -- (3,3) arc(360:180:1.5) -- (0,5);
\draw (1.5,1.3) node[anchor=north] {$\Omega_0$};
\fill[blue!25!white] (4.75,5) -- (5.75,5) -- (5.75,3) arc(180:225:.5) arc(315:360:.5) -- (6.05,5) -- (7.05,5) -- (7.05,3) arc(360:346:1.5) .. controls (7.2,2.6) .. (7.3,2.5) arc (340:200:1.5) .. controls (4.67,2.6) .. (4.79,2.65) arc (194:180:1.5) -- (4.75,5);
\draw[line width=1] (4.75,5) -- (5.75,5) -- (5.75,3) arc(180:225:.5) arc(315:360:.5) -- (6.05,5) -- (7.05,5) -- (7.05,3) arc(360:346:1.5) .. controls (7.2,2.6) .. (7.3,2.5) arc (340:200:1.5) .. controls (4.67,2.6) .. (4.79,2.65) arc (194:180:1.5) -- (4.75,5);
\draw (5.9,1.3) node[anchor=north] {$\Omega_\sigma$};
\draw (10.3,1.3) node[anchor=north] {$\Omega_\tau$};
\fill[blue!25!white] (9.3,5) --(10.3,5) -- (10.3,3) -- (10.3,5) -- (11.3,5) -- (11.3,3) .. controls (11.5,2.85) ..  (11.7,2.5) arc (340:200:1.5) .. controls (9.1,2.85) .. (9.3,3) -- (9.3,5);
\draw[line width=1] (9.3,5) --(10.3,5) -- (10.3,3) -- (10.3,5) -- (11.3,5) -- (11.3,3) .. controls (11.5,2.85) ..  (11.7,2.5) arc (340:200:1.5) .. controls (9.1,2.85) .. (9.3,3) -- (9.3,5);
\draw[dashed] (-.5,2.5)-- (12,2.5);
\draw[dashed] (-.5,2.65)-- (12,2.65);
\draw[dashed] (-.5,3)-- (12,3);
\draw (-0.5,2.5) node[anchor=north east] {$y_0$};
\draw (-0.5,2.75) node[anchor=east] {$y_\sigma$};
\draw (-0.5,2.95) node[anchor=south east] {$y_\tau$};
\end{tikzpicture}
\caption{A set $\Omega$ such that $t\mapsto \lambda(\Omega_t)$ is discontinuous.}\label{Figdisc}
\end{figure}
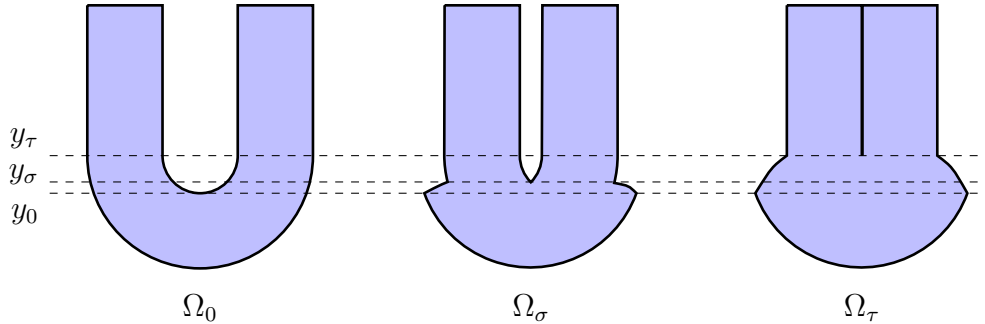
The reason why the sets behave badly if $s\searrow t$ can be easily understood with an example. Let us assume that $\Omega=\Omega_0$ has a U-shape as in Figure~\ref{Figdisc}, and that $\nu$ is the horizontal vector. The set $\Omega$ already coincides with $\Omega^*_\nu$ below a height $y_0$, hence for every $t>0$ the sets $\Omega_0,\, \Omega_t$ and $\Omega_\infty=\Omega^*_\nu$ coincide below this height. For a small time $\sigma>0$, the two ``legs'' of $\Omega$ have become closer, and they have already met below a height $y_\sigma$, hence below this height all the sets $\Omega_t$ coincide for $t>\sigma$. There is then a particular time $\tau$ when the two internal, vertical segments in the boundary of $\Omega_\tau$ coincide. Notice that the set $\Omega^+_\tau$ defined above consists in the set $\Omega_\tau$ together with the internal, vertical segment, and actually $\Omega_t=\Omega^+_\tau=\Omega_\infty=\Omega^*_\nu$ for every $t>\tau$. It is obvious that the functions $t\mapsto\lambda(\Omega_t)$ and $t\mapsto T(\Omega_t)$ are continuous for $0\leq t\leq\tau$, and according with Brock's result they are also respectively decreasing and increasing. After the time $\tau$, instead, since the vertical segment suddenly disappears, there is clearly a jump in both functions.


\section{The case of the polyhedra\label{pfmain}}

This section is devoted to consider the case of polyhedra, and to show Theorem~\ref{main}. The idea is simple; if $\Omega_0$ is a polyhedron then, similarly to what happens in the example considered in Figure~\ref{Figdisc}, the path $t\mapsto \Omega_t$ is already $\gamma$-continuous, except at finitely many instants where a $(d-1)$-dimensional wall suddenly disappears. It is then sufficient to modify the construction letting these ``walls'' smoothly disappear in a positive time, gaining then the $\gamma$-continuity.

\begin{proof}[Proof (of Theorem~\ref{main})]
Let $\Omega\subset\R^d$ be a polyhedron, and let $\nu\in\S^{d-1}$ be a given direction. Notice that also the set $\Omega^*_\nu$ is a polyhedron. As already said in the introduction, for every open set $A$ compactly contained in $D$ we call $u_A$ the torsion function, i.e., the unique solution of the PDE $-\Delta u =1$ in $H^1_0$, extended by $0$ in $D\setminus A$. Moreover, for every $t>0$, we define $\Omega_t^+={\rm Int}\big(\overline{\Omega_t}\big)$.\par

Let $t\geq 0$ be any positive number, and let $s_n\searrow t$ be a sequence converging to $t$ from above. The functions $u_{\Omega_n}$ form a bounded sequence in $H^1_0(D)$, hence a subsequence converges to some function $\bar u$ weakly in $H^1_0(D)$, so in particular strongly in $L^2(D)$. It is simple to observe that, since $\Omega$ is a polyhedron, $\bar u$ belongs to $H^1_0(\Omega_t^+)$. Here the assumption that $\Omega$ is a polyhedron is essential, since examples show that this assertion is in general false, even with the assumption that $\Omega$ is a smooth open set! Notice that
\[
T(\Omega_t^+) \geq \frac{\bal\bigg(\int\bar u\,dx\bigg)^2\eal}{\bal\int|\nabla\bar u|^2\,dx\eal} \geq \limsup_{n\to\infty} \frac{\bal\bigg(\int u_n\,dx\bigg)^2\eal}{\bal\int|\nabla u_n|^2\,dx\eal}
=\limsup_{n\to\infty} T(\Omega_n) \geq T(\Omega_t^+)\,.
\]
The last inequality is true because the torsional rigidity increases with time and, by definition, for every $s>t$ one has that $\Omega_s=(\Omega_t)_{s-t}=(\Omega_t^+)_{s-t}$. By the above chain of inequalities, and by the uniqueness of the torsion function, we deduce that $\bar u=u_{\Omega_t^+}$. Therefore, since the convergence of $u_{\Omega_n}$ to $u_\Omega=\bar u$ is strong also in $L^1$, we deduce that the sets $\Omega_s$, when $s\searrow t$, $\gamma$-converge to $\Omega_t^+$.\par

Observe that, by construction, $\Omega_t$ is an open set contained in $\Omega_t^+$. Moreover, they have the same measure thanks to Fubini Theorem, since for every $y \in \nu^\perp$ the difference $(\Omega_t^+\setminus \Omega_t)^y$ has only finitely many points. Therefore, by the maximum principle we have $u_{\Omega_t}\leq u_{\Omega_t^+}$, thus
\[
d_\gamma(\Omega_t,\Omega_t^+) = \|u_{\Omega_t}-u_{\Omega_t^+}\|_{L^1} = \int\big(u_{\Omega_t^+}-u_{\Omega_t}\big)\,dx = T(\Omega_t^+)-T(\Omega_t).
\]
Again using the fact that $\Omega$ is a polyhedron, there can be at most finitely many instants $t_1<t_2< \cdots < t_N$ such that the above difference is strictly positive, thus the path $t\mapsto \Omega_t$ is already $\gamma$-continuos in $\R\setminus \{t_1,\,t_2,\, \dots\, , \, t_N\}$.\par

Let now $t$ be any of the instants $t_j$. For every $0\leq \eta\leq 1$ we can define a set $\Omega_{t,\eta}$, ranging from $\Omega_{t,0}=\Omega_t$ to $\Omega_{t,1}=\Omega_t^+$. The sets $\Omega_{t,\eta}$ are defined continuously increasing, i.e., continuously shrinking the ``wall'' $\Omega_t^+\setminus \Omega_t$. Since for every $\eta< \xi$ we have $\Omega_{t,\eta}\subset \Omega_{t,\xi}$, then as before we obtain
\[
d_\gamma(\Omega_{t,\eta},\Omega_{t,\xi}) = T(\Omega_{t,\xi})-T(\Omega_{t,\eta}).
\]
Therefore, the path $\eta\mapsto \Omega_{t,\eta}$ is $\gamma$-continuous. Since the sets are increasing, then the first eigenvalue and the torsional rigidity are respectively decreasing and increasing, in a continuous way since both quantities are $\gamma$-continuous.\par

It is now clear how to modify the definition of the sets $\Omega_t$, replacing every instant $\{t_j\}$ with a closed time interval of width $1$, in such a way the map $[0,+\infty]\ni t\mapsto \Omega_t$ is a $\gamma$-continuous path between $\Omega$ and $\Omega^*_\nu$ and the first eigenvalue and the torsional rigidity are monotone (respectively decreasing and increasing) and continuous.\par

It is then sufficient to perform the same construction countably many times in different directions, so to eventually obtain a family of sets that $\gamma$-converge to a ball. By reparametrizing the variable $t$, we can let it vary in the closed interval $[0,1]$.
\end{proof}

\section{Application to the Blaschke-Santal\'o diagram\label{ssantalo}}

The study of Blaschke-Santal\'o diagrams is a very powerful way to treat shape optimization problems, which are in general rather difficult to attack because the class of admissible shapes do not have strong functional properties and very often limits of sequences of shapes (in particular $\gamma$-limits) are not shapes any more. If $A(\O)$ and $B(\O)$ are two shape functionals (a similar argument can be used for a larger number of them) many shape optimization problems can be written in the form
\be\label{shopt}
\min\big\{F\big(A(\O),B(\O)\big)\ :\ |\O|=m\big\},
\ee
where the Lebesgue measure constraint is very natural in this kind of problems. Sometimes, the presence of additional geometric constraints (as for instance convexity of admissible shapes or other geometric bounds a priori imposed) makes the above problem easier, since extra compactness properties can be deduced. When the quantities $A(\O)$ and $B(\O)$ fulfill suitable scaling relations as
\begin{align*}
A(t\O)=t^\alpha A(\O)\,,&& B(t\O)=t^\beta B(\O)\,,
\end{align*}
and if the function $F$ is expressed through powers, as
\[
F(A,B)=A^p B^q\,,
\]
the Lebesgue measure constraint $|\O|=m$ can be incorporated in the {\it scaling free} functional
\[
\F(\O)=\frac{A^p(\O) B^q(\O)}{|\O|^{(\alpha p+\beta q)/d}}=\left(\frac{A(\O)}{|\O|^{\alpha/d}}\right)^p\left(\frac{B(\O)}{|\O|^{\beta/d}}\right)^q\,,
\]
and the minimum problem above can be reformulated as the minimum problem for $\F$ without any Lebesgue measure constraint.

The Blaschke-Santal\'o diagram for the pair $A(\O)$, $B(\O)$ is the subset of the Euclidean space $\R^2$ given by
\[
E=\left\{(x,y)\in\R^2\ :\ x=\frac{A(\O)}{|\O|^{\alpha/d}},\ y=\frac{B(\O)}{|\O|^{\beta/d}}\text{ for some }\O\right\}\,.
\]
In this way our shape optimization problem~\eqref{shopt} can be reduced to the optimization problem on $\R^2$ given by
\[
\min\big\{F(x,y)\ :\ (x,y)\in E\big\}\,.
\]
In general the full characterization of the Blaschke-Santal\'o diagram $E$ is a difficult problem and often only some bounds can be obtained. In the present paper we consider the quantities $\lambda(\O)$ and $T(\O)$ and we try to identify the set $E$ in this case. In order to have the set $E$ included in the square $[0,1]\times[0,1]$ it is convenient to take the rescaled variables
\begin{align}\label{xy}
x=\frac{|B|^{2/d}\lambda(B)}{|\O|^{2/d}\lambda(\O)}\,, && y=\frac{|B|^{(d+2)/d}T(\O)}{|\O|^{(d+2)/d}T(B)}\,,
\end{align}
being $B$ a ball of radius $1$. In this way the Kohler-Jobin inequality (see for instance~\cite{BBP20})
\begin{equation}\label{kji}
\lambda(\O)T^{2/(d+2)}(\O)\ge\lambda(B)T^{2/(d+2)}(B)
\end{equation}
becomes, in the $x,y$ variables,
\be\label{kjxy}
y\ge x^{(d+2)/2}\,.
\ee
Instead, the Polya inequality $\lambda(\O)T(\O)<|\O|$ (see~\cite{BBP20}) becomes
\[
y<\frac{|B|}{\lambda(B)T(B)}\, x\,.
\]
A slight improvement of this inequality has been obtained in~\cite{bfnt16}, where it is proved that
\[
\lambda(\O)T(\O)\leq |\O| \bigg(1-\frac{2d|B|^{2/d}}{d+2}\frac{T(\O)}{|\O|^{(d+2)/d}}\bigg)\,,
\]
which, by~(\ref{xy}) and since a simple calculation ensures $T(B)=\omega_d/\big(d(d+2)\big)$, gives
\be\label{bfntxy}
y\le \frac{|B| x}{\lambda(B)T(B)}\left(1-\frac{2xd}{2xd+(d+2)\lambda(B)}\right)\,,
\ee
In Figure~\ref{fig1} we plot the bounds~(\ref{kjxy}) and~(\ref{bfntxy}) in the case of dimension two, which are
\[
x^2\le y\le\frac{8x}{x+j_0^2}\,,
\]
being $j_0=2.4048\dots$ the first zero of the Bessel function $J_0$.

\begin{figure}[htbp]
\centering{\includegraphics[scale=1.2]{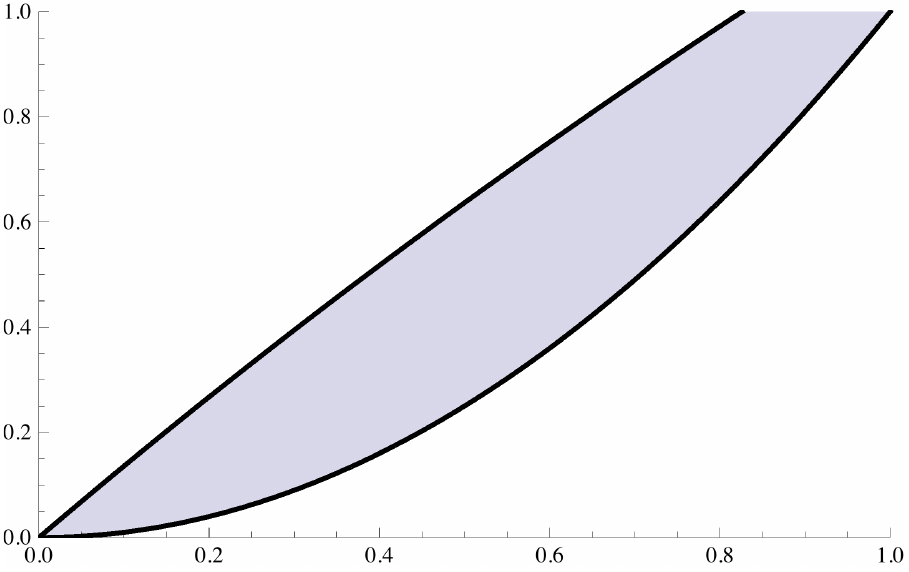}}
\caption{The colored region, obtained by the inequalities~\eqref{kjxy} and~\eqref{bfntxy}, contains the Blaschke-Santal\'o diagram $E$ for $\lambda(\O)$ and $T(\O)$ in the case $d=2$.}\label{fig1}
\end{figure}

We start to study some properties of the Blaschke-Santal\'o diagram $E$.

\begin{lemm}\label{ncurves}
For every $(x_0,y_0)\in E$ there exists a sequence of continuous curves $\big(x_n(\sigma),y_n(\sigma)\big)$ in $E$, with $\sigma\in[0,1]$, such that $\big(x_n(0),y_n(0)\big)=(x_0,y_0)$, converging uniformly to the curve
\begin{align*}
x(\sigma)=(1-\sigma)^2x_0\,, && y(\sigma)=(1-\sigma)^{d+2}y_0\,,&& \sigma\in[0,1]
\end{align*}
which connects the point $(x_0,y_0)$ with the origin. In Cartesian coordinates the limit curve is the graph of the function
\begin{align*}
y=y_0(x/x_0)^{(d+2)/2} && x\in[0,x_0]\,.
\end{align*}
\end{lemm}
\begin{proof}
Let $\O$ be a domain which gives the point $(x_0,y_0)\in E$, that is
\begin{align*}
x_0=\frac{|B|^{2/d}\lambda(B)}{|\O|^{2/d}\lambda(\O)}\,, && y_0=\frac{|B|^{(d+2)/d}T(\O)}{|\O|^{(d+2)/d}T(B)}\,.
\end{align*}
For every $n$ let $a_n=1-n^{-1/d}$ and, for $\sigma\in[0,1]$, let $\O^n_\sigma$ be the domain which consists of the union of $(1-a_n\sigma)\O$ and $n-1$ disjoint copies of $\left(\frac{1-(1-a_n\sigma)^d}{n-1}\right)^{1/d}\O$. We have $|\O^n_\sigma|=|\O|$ and
\[\begin{cases}
\lambda(\O^n_\sigma)=(1-a_n\sigma)^{-2}\lambda(\O)\\
T(\O^n_\sigma)=\left[(1-a_n\sigma)^{d+2}+(n-1)^{-2/d}\big(1-(1-a_n\sigma)^d\big)^{(d+2)/d}\right]T(\O).
\end{cases}\]
In terms of $(x,y)$ variables we have the curve
\begin{align*}\begin{cases}
x_n(\sigma)=x_0(1-a_n\sigma)^2\\
y_n(\sigma)=y_0\left[(1-a_n\sigma)^{d+2}+(n-1)^{-2/d}\big(1-(1-a_n\sigma)^d\big)^{(d+2)/d}\right]
\end{cases} && \sigma\in[0,1]\,,
\end{align*}
or, in Cartesian coordinates,
\begin{align}\label{equphin}
y=y_0\left[(x/x_0)^{(d+2)/2}+(n-1)^{-2/d}\big(1-(x/x_0)^{d/2}\big)^{(d+2)/d}\right]&& x/x_0\in[(1-a_n)^2,1]\,.
\end{align}
It is immediate to see the uniform convergence of the sequence of curves $\big(x_n(\sigma),y_n(\sigma)\big)$ to the limit curve
\begin{align*}
x(\sigma)=(1-\sigma)^2x_0\,, && y(\sigma)=(1-\sigma)^{d+2}y_0,\qquad \sigma\in[0,1]\,,
\end{align*}
as required.
\end{proof}

We are now in a position to prove our result concerning the structure of the Blaschke-Santal\'o diagram $E$ of all points $(x,y)\in\R^2$ with $x$ and $y$ given by~\eqref{xy}.

\begin{proof}[Proof (of Theorem~\ref{bsgraph})]
In order to prove the existence of an increasing function $h$ satisfying~(\ref{openincl}) it is enough to show that, for every $(x_0,y_0)\in E$, all the points $(x,y)\in [0,1]^2$ with $y> x^{(d+2)/2}$ and with $x>x_0,\, y<y_0$ are also contained in $E$. To obtain this convexity property we rely on Theorem~\ref{main} and Lemma~\ref{ncurves}. More precisely, let $(x_0,y_0)\in E$, and let us first assume that it corresponds via~(\ref{xy}) to a polyhedron $\Omega$. Let then $\Omega_t$, with $t\in [0,1]$, be the $\gamma$-continuous map given by Theorem~\ref{main}, and let $\varphi:[0,1]\to E$ be the map given by $\varphi(t)=(x_t,y_t)$, where $(x_t,y_t)$ is given by~(\ref{xy}) with $\Omega_t$ in place of $\Omega$. By Theorem~\ref{main}, $\varphi$ is a curve which continuously connects $(x_0,y_0)$ with $(1,1)$, and which is increasing in both variables. For every $0\leq t\leq 1$, by Lemma~\ref{ncurves} we have a sequence of continuous curves, explicitely given by~(\ref{equphin}), all starting from $(x_t,y_t)$ and uniformly converging to the graph of $x\mapsto y_t (x/x_t)^{(d+2)/2}$, $x\in [0,x_t]$. A very simple continuity argument, graphically depicted in Figure~\ref{fig2}, implies then that all the points $(x,y)$ with $x_0<x<1$ and $x^{(d+2)/2}<y<y_0$ belong to $E$.
\begin{figure}[htbp]
\begin{tikzpicture}
\draw[line width=1] (0,0) rectangle (10,8);
\draw (0,0) node[anchor=north east] {$(0,0)$};
\draw (10,0) node[anchor=north west] {$(1,0)$};
\draw (0,8) node[anchor=south east] {$(0,1)$};
\draw (10,8) node[anchor=south west] {$(1,1)$};
\draw[line width=1] (0,0) parabola (10,8);
\fill (4,4) circle (2.5pt);
\draw (4,4) node[anchor=south east] {$(x_0,y_0)$};
\draw (4,4) .. controls (5.2,8) and (8,5) .. (10,8);
\draw (6.5,6.35) node[anchor=south] {$\varphi$};
\fill (5,5.76) circle (1.5pt);
\draw[dashed] (0,0) parabola (5,5.76);
\fill (6,6.25) circle (1.5pt);
\draw[dashed] (0,0) parabola (6,6.25);
\fill (7,6.4) circle (1.5pt);
\draw[dashed] (0,0) parabola (7,6.4);
\fill (8,6.55) circle (1.5pt);
\draw[dashed] (0,0) parabola (8,6.55);
\end{tikzpicture}
\caption{Argument of the proof of Theorem~\ref{bsgraph}.}\label{fig2}
\end{figure}
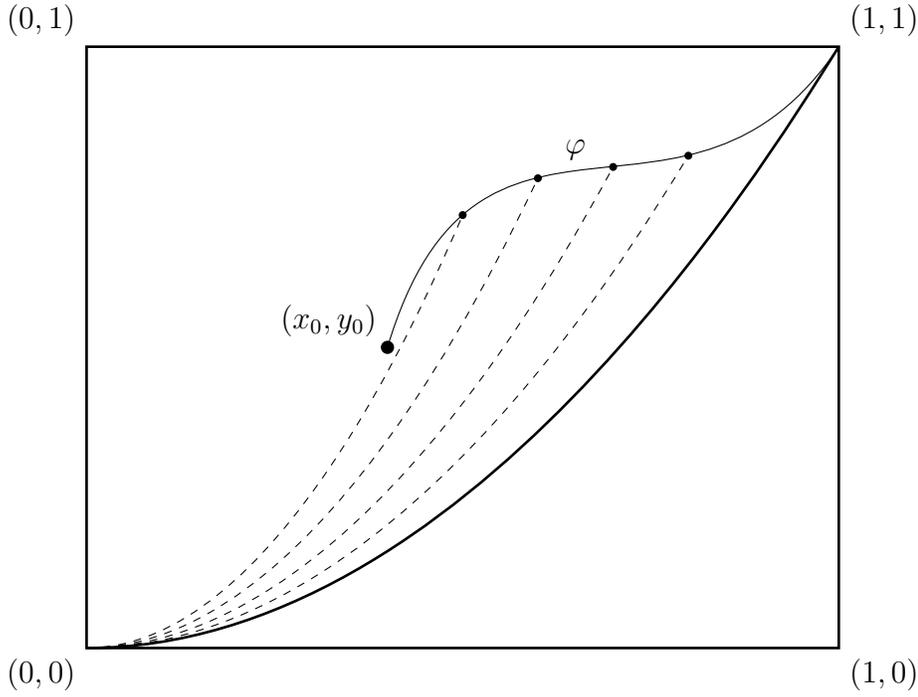
Let us now take a generic point $(x_0,y_0)\in E$, corresponding to an open domain $\Omega$. Let $\{\Omega_k\}_{k\in\N}$ be a sequence of polyhedra which approximate $\Omega$ from inside, hence which $\gamma$-converge to $\Omega$. If we call $(x_k,y_k)$ the numbers given by~(\ref{xy}) with $\Omega_k$ in place of $\Omega$, we have then that the points $(x_k,y_k)$ converge to $(x_0,y_0)$. The argument already presented for polyhedra ensures that every pair $(x,y)\in [0,1]^2$ such that $y> x^{(d+2)/2}$ and such that $x>x_k$ and $y<y_k$ for some $k\in\N$ belongs to $E$. Of course, if $x>x_0$ and $y<y_0$ then $x>x_k$ and $y<y_k$ for $k$ large enough, hence the existence of an increasing function $h$ satisfying~(\ref{openincl}) follows.\par

Finally, concerning the bound~(\ref{estih}) on $h$, the upper one coincides with~\eqref{bfntxy}, and the lower one is proved in~\cite[Proposition~7.2]{BBP20}.
\end{proof}

\begin{rema}\label{remoi}
We conclude with a short discussion about the equalities in~(\ref{openincl}). More precisely, it would be interesting to determine whether or not the points $(x,y)\in [0,1]^2$ with $y=x^{(d+2)/2}$ or with $y=h(x)$ belong to $E$. The first part is actually known. Indeed, as observed in~\cite[Remark~4.2]{Br13}, the Kohler-Jobin inequality~(\ref{kji}) is strict for every set $\Omega$ which is not a ball. Therefore, the point $(x,x^{(d+2)/2})$ does not belong to $E$ for every $0\leq x<1$, while of course $(1,1)\in E$, since it corresponds to the ball. Instead, we do not know whether the points $(x,h(x))$ belong to $E$ for $0<x<1$.
\end{rema}

\bigskip
\noindent{\bf Acknowledgments. }This work is part of the project 2017TEXA3H {\it``Gradient flows, Optimal Transport and Metric Measure Structures''} funded by the Italian Ministry of Research and University. The authors are member of the Gruppo Nazionale per l'Analisi Matematica, la Probabilit\`a e le loro Applicazioni (GNAMPA) of the Istituto Nazionale di Alta Matematica (INdAM).

\bigskip
{\small\noindent
Giuseppe Buttazzo:
Dipartimento di Matematica,
Universit\`a di Pisa\\
Largo B. Pontecorvo 5,
56127 Pisa - ITALY\\
{\tt giuseppe.buttazzo@dm.unipi.it}\\
{\tt http://www.dm.unipi.it/pages/buttazzo/}

\bigskip\noindent
Aldo Pratelli:
Dipartimento di Matematica,
Universit\`a di Pisa\\
Largo B. Pontecorvo 5,
56127 Pisa - ITALY\\
{\tt aldo.pratelli@dm.unipi.it}\\
{\tt http://pagine.dm.unipi.it/pratelli/}}

\end{document}